\documentclass[a4paper,11pt,twoside]{article}
\usepackage[T1]{fontenc}
\usepackage[latin9]{inputenc}
\usepackage{amsmath,amssymb,amsthm,graphicx,nicefrac}
\usepackage{microtype,url}

\usepackage{hyperref}
\hypersetup{colorlinks=false,pdfborder=0 0 0.2}

\newtheorem{prop}{Proposition}[section]
\newtheorem{coro}[prop]{Corollary}
\newtheorem{thm}[prop]{Theorem}

\newtheorem{defi}[prop]{Definition}
\newtheorem{ansatz}[prop]{Ansatz}

\makeatletter
\@addtoreset{equation}{section}
\makeatother
\def\DD{\displaystyle}
\def\egaldef{\stackrel{\mbox{\tiny def}}{=}}
\newcommand\C{\mathcal{C}}
\newcommand\D{\mathcal{D}}
\newcommand\B{\mathcal{B}}

\newcommand\R{\mathcal{R}}
\newcommand\LL{\mathbb{L}}

\renewcommand\L{\mathcal{L}}
\newcommand\n{^{\scriptscriptstyle (N)}}
\renewcommand\>[1]{\vec{#1}}

\begin{document}
\title{Modeling a Case of Herding Behavior in a Multi-Player Game\thanks{This work was supported
    by grant 109-2167/R from the \emph{Région Île-de-France}}}
\author{Guy Fayolle\thanks{INRIA Paris--Rocquencourt --- Domaine de
 Voluceau B.P.\ 105, 78153 Le Chesnay cedex (France)} \and Jean-Marc Lasgouttes\footnotemark[2]}
\date{\today }

\maketitle

\begin{abstract}
\noindent
The system mentioned in the title belongs to the family of the
so-called massively multi-player online social games (MMOSG). It
features a scoring system for the elements of the game that is
prone to herding effects. We analyze in detail its stationary regime
in the thermodynamic limit, when the number of players tends to
infinity. In particular, for some classes of input sequences and
selection policies, we provide necessary and sufficient conditions for
the existence of a complete meanfield-like measure, showing off an
interesting \emph{condensation} phenomenon.
\end{abstract}


\smallskip\noindent
\emph{AMS $2010$ Subject Classification}: primary 60J28; secondary 34L30, 37C40.

\section{Description of the basic model}\label{sec:introduction}
\emph{Ma Micro Plan\`ete}~\cite{MMP} is a geolocalized massively
multi-player online social game which entices players to use
sustainable means of transport. At the heart of the game is a
community-driven creation of a set of \emph{points of interest}
(POIs): interesting places, events, or even traffic incidents. The
popularity of the POIs is assessed by scores. At his turn, the player
is presented with a number of POIs to visit in his neighborhood, which
have been selected through a mechanism that favors already popular
ones. Then the selected POI has its score increased by a random value.
It is not difficult to see that there is a risk of \emph{herding}
behavior, that would concentrate all the attention on a few places
only. The aim of this paper is to present some results about the POIs
dynamics by means of a probabilistic model, and in particular to
determine when a stationary solution exists.

Assume there are $N$ players in the game and let $n_i(t;N)$ be the
number of POIs having an integer score $i$ at time $t$. The state of
the system at time $t$ can be described by the infinite sequence
\[
 \>R(t;N) = \frac{1}{N} [n_1(t;N), \ldots, n_i(t;N),\ldots].
\]

The dynamics of the model is
\begin{itemize}
\item POIs are created according to a Poisson process of rate $\lambda
N$. Each new POI is granted an initial score $k$ with probability
$\varphi_k,k\ge1$, where the $\varphi_k$'s form a proper
probability distribution.
\item For any $i\ge1$, each POI with score $i$ decreases its score by
$1$ after a random time exponentially distributed at rate $\mu$, where
$\mu$ is a strictly positive constant.
\item The $N$ players are all independent and visit POIs at instants
forming a global Poisson process with rate $\alpha N$. Upon being
visited, a POI sees its score increased by $j\ge1$ with probability
$\theta_j$, independently of any other event, where the sequence
$\{\theta_j,j\ge1\}$ forms a proper probability distribution. We shall
implicitly take $\theta_0=0$, which is in no way a restriction, as it
is simply tantamount to modifying the parameter $\alpha$.
\item When he decides to play, at the instant of a Poisson process of
parameter $\alpha N$, a player selects a POI having score $i\ge1$ with
probability $\pi_i(\>R(t;N))$, written as a functional of the state of
the system, with the normalizing condition
\[
\sum_{i\ge1} \pi_i(\>R(t;N)) =1.
\]
\end{itemize}

In accordance with the description of the game, we shall essentially
consider selection policies such that the probability $\pi$ of
selecting a POI, given its score $i$, is a \emph{non-decreasing}
function of $i$.

The following typical  scoring rules are used in \emph{Ma Micro Plan\`ete}.
\begin{itemize}
\item Each player starts a journey at rate $\alpha$ and selects a POI
to visit with a probability proportional to its current score.
\begin{itemize}
\item with probability $p$, the player adds new information to
the POI, the score of which is increased by $15$ points;
\item otherwise, the POI is merely visited and its score
gains $5$ points; this happens with probability $1-p$.
\end{itemize}
\item Each player creates POIs at rate $\lambda$ and their
initial score is set to $50$.
\item The POIs decrease their scores by $1$ at a rate $\mu=3$ per day.
\item A POI disappears as soon as its score becomes zero
or negative.
\item The values of $\alpha$, $p$ and $\lambda$ can only be measured
empirically from concrete game statistics.
\end{itemize}

A way of describing this model might be in terms of queues. Consider a
possibly infinite array of queues $M/M/1$ queues with service rate
$\mu$. New queues are created at rate $\lambda$ and batches of
customers join an existing queue with score $i\ge1$ at rate
$\alpha\pi_i(\>R(t;N))$. Such systems have already been described, for
example when the customers join the shortest of two
queues~\cite{VveDobKar}. The situation described in this paper is
unusual due to the bias towards selecting the \emph{longest} queue.
This only makes sense because there is no cost associated to visiting
a highly popular POI, whereas joining a long queue is expensive in
terms of waiting time. A similar situation~\cite{VeeDeb} occurs when
customers, unsure of the quality of two restaurants, choose the busier
one.

The model can have ramifications with other areas of interest, e.g.,
dynamics of populations, social networks, chemical or energy networks
in physics.

In this work, we examine the system in the so-called
\emph{thermodynamic limit}, as $N\to\infty$, the existence of which is
proved in Section~\ref{sec:thermo}, under general conditions, where
the probability of increasing the score of a visited POI is a function
of the state of the system. Section~\ref{sec:stat} is devoted to the
stationary regime, for which complete (but surprising!) answers are
given for various input sequences and selection policies, and the
existence of a so-called herding behavior is connected to possible
condensation phenomena. Section~\ref{sec:examples}
presents miscellaneous simple case studies. The final
Section~\ref{sec:general} outlines a possible method of solution for
general input sequences, and also discusses a brief list of unsolved
questions.

\section{The thermodynamical limit}\label{sec:thermo}
To capture information on the behavior of the system under various
game policies, we proceed by scaling and analyze the
\emph{thermodynamical limit}, when the number $N$ of players becomes
large. In mathematical terms, we consider the system when
$N\to\infty$, for any fixed $t$.
\subsection{Notation}
It will be convenient to gather here most of the notational material
concerning input sequences and mathematical symbols.
\begin{itemize}
\item $\mathcal{E} \egaldef \{\>y = [y_1,\dots,y_i,\ldots], y_i\in
\mathbb{Q}\}$, where $\mathbb{Q}_+$ is the set of positive rational
numbers. $\mathbb{R}^\infty_+$ being the set of non-negative real
numbers, $\>e_j$ will stand for the $j$-th unit vector of
$\mathbb{R}^\infty_+$. We shall also denote by
$\LL\in\mathbb{R}^\infty_+$ the subset of bounded sequences with the norm 
$\|\>y\| \egaldef \sum_{i\ge1} |y_i|< M<\infty$.
\item $E$ being an arbitrary metric space, $\B(E)$ is the space of
 bounded real-valued Borel measurable functions on $E$,
 $\C(E)\subset\B(E)$ the Banach subspace of bounded continuous functions.
\item We shall need the three following subspaces of $\C(\LL)$:
 \begin{eqnarray*}
\C_1(\LL) & \egaldef &\bigg\{ f : \sup_{j\ge1}\, \left|\frac{\partial f}{\partial y_j} \right| \leq C \biggr\}, 
\\[0.1cm]
 \C_2(\LL) & \egaldef & \biggl\{f : \sup_{j,k\geq 1}\, \biggl| \frac{\partial^2 f}{\partial y_j\partial y_k}
\biggr| \leq H  \bigg\},
\end{eqnarray*}
$C,H$ being arbitrary positive constants.
\item For any complete separable metric space $S$, let $D_S[0,\infty]$
denote the space of right continuous functions $f:[0,\infty]
\rightarrow S$ with left limits, endowed with the Skorokhod topology.
\end{itemize}

\subsection{The main limit theorem}\label{th:thermo}
Clearly $\>R(t;N)$ is a continuous time Markov chain, and its
generator $G\n$, for $f\in\C_1(\LL)$, is the operator given by,
\begin{align}\label{eq:gen}
G\n f(\>y) &= \lambda N \sum_{j\ge1} \varphi_j \left[f\left(\>y+\frac{\>e_j}{N}\right) - f(\>y)\right]\nonumber\\[0.2cm]
	&+\mu N \sum_{j\ge1} y_j \left[ f\left(\>y+\frac{\>e_{j-1}}{N} - \frac{\>e_j}{N}\right) - f(\>y)\right]
	\nonumber\\[0.2cm]
		&+ \alpha N \sum_{j\ge1} \pi_j(\>y)\sum_{i\ge1}\theta_i  \left[f\left(\>y - \frac{\>e_j}{N}
	+ \frac{\>e_{i+j}}{N}\right) - f(\>y)\right].
\end{align}
Throughout this study, the game will be subject to the following
reasonably weak assumption.
\paragraph{Assumption (L)} \emph{For all $\>r\in\LL$}, the
probability distribution $\pi: \LL\to\LL$, considered as an infinite
vector $\{\pi_i( \>r), i\ge1\}$ with values in the Banach space
$\LL$ of absolutely summable sequences, will be supposed to be
boundedly and twice continuously differentiable (see e.g.~\cite{Car}).
In this case, all partial derivatives
$\frac{\partial\pi_i(\>y)}{\partial y_j}, j,k\ge1, \>y\in\LL$, belong
to $\C_1(\LL)$ and the famous Lipschitz condition is also satisfied.

 \smallskip In particular the above assumption (\textbf{L}) holds  for
the forthcoming policy, denoted by $\textbf{P}_{\>a}$, which will be
analyzed in Section~\ref{sec:stat},
\[
\pi_i(\>r)= \frac{a_ir_i}{\sum_{j\ge1}a_jr_j},
\]
where the $a_i$'s form a bounded sequence of positive numbers.

Then, with the notation $\theta_0=\pi_0(.)=0$, the  following proposition holds.
\begin{thm}\label{thm:thermo}
Suppose that, as $N\to\infty$, $\>R(0;N)$ converges in
distribution to $\>R(0)\in\LL$.
Then the Markov process $\>R(t;N)$ converges also weakly in
$D[0,\infty]$ to a \emph{deterministic} dynamical system
$\>r(.)\in\mathbb{R}^\infty_+$, which evolves according to the
following infinite system of nonlinear differential equations
\begin{equation}\label{eq:sys}
\frac{\mathrm{d}r_i(t)}{\mathrm{d}t}
  =  \lambda\varphi_i+\mu \bigl[r_{i+1}(t)-r_i(t)\bigr] +
  \alpha \bigl[\sum_{j=0}^{j=i} \theta_j \pi_{i-j}(\>r(t))-\pi_i(\>r(t))\bigr],
\end{equation}
where $\pi_i(\>r(t))$ denotes the probability that the currently
visited POI has score~$i$. Moreover, under Assumption (\textbf{L}) and
for any given finite initial condition $\>r(0)\in\LL$, the
system~\eqref{eq:sys} admits a unique solution, for all $t\ge0$.
\end{thm}

\begin{proof}
The method to solve this weak convergence problem is somehow classical
(see for example the footsteps of~\cite{DelFay}), and relies on
theoretical results of~\cite{EK}. First, a second order Taylor's
expansion in~\eqref{eq:gen} yields immediately, for any $f\in
\C_2(\LL)$,
\begin{equation}\label{eq:limgen}
G\n f(\>y) = Gf(\>y) + \mathcal{O}\left(\frac{1}{N}\right),
\end{equation}
where $G$ is the operator with domain $\C_1(\LL)$ satisfying
  \begin{align}
 Gf(\>y) & = \lambda \sum_{j\ge1} \varphi_j \frac{\partial f}{\partial y_j} +
 \mu \sum_{j\ge1} y_j  \left(\frac{\partial f}{\partial y_{j-1}} - \frac{\partial f}{\partial y_j}\right)(\>y)     \nonumber \\[0.2cm]
 & + \alpha  \sum_{j\ge1} \pi_j(\>y) \sum_{i\ge1}\theta_i \left(\frac{\partial f}{\partial y_{i+j}} - \frac{\partial f}{\partial y_j}\right)(\>y).\label{eq:gen2}
\end{align}
Indeed, taking for instance the second term in the right-hand side of
\eqref{eq:gen}, we have
 \[
\mu N \sum_{j\ge1} y_j \!\left[ f\left(\>y+\frac{\>e_{j-1}}{N} - \frac{\>e_j}{N}\right) - f(\>y)\right] =
 \mu \sum_{j\ge1} y_j  \left(\frac{\partial f}{\partial y_{j-1}} - \frac{\partial f}{\partial y_j}\right)(\>y) 
+ \mathcal{R}_N,
\]
where, by Taylor's formula,
\[
\R_N = \frac{\mu}{N}  \sum_{j\ge1} y_j\biggl(\frac{\partial^2f}{\partial y^2_{j-1}} + 
\frac{\partial^2 f}{\partial y^2_j} - 2  \frac{\partial^2 f}{\partial y_j\partial y_{j-1}}\bigg)(\>x_j),
\]
and
\[
\>x_j = \>y_j  + \delta_j \left(\frac{\>e_{j-1}}{N} - \frac{\>e_j}{N}\right), \quad 0<\delta_j<1.
\]
So, for any $f\in\C_2(\LL)$, we get  the estimate 
\[
|\R_N| \le \frac{4\mu MH}{N}.
\]
The other terms in~\eqref{eq:gen} could be dealt with in a similar way.

Then, for $\>r\in\LL$, denote by $U(\>r)$ the right-hand side of
system~\eqref{eq:sys}. Under condition \textbf{L}, it is known (see
e.g.~\cite{Car,Ince}) that the solution $\>x : [0,\infty)\times
\LL\to\LL$ of the differential system
\begin{equation}\label{eq:diffx}
\begin{cases}
\DD\frac{d\>x(t,\>a)}{dt} = U(\>x(t,\>a)), \\[0.2cm]
\>x(0,\>a) = \>a,
\end{cases}
\end{equation}
is unique, for any initial condition $\>a\in\LL$ and all finite $t$.
Therefore, for $f\in\C(\LL)$, one can define a one-parameter
family $\{T(t),t\ge0\}$, such that
 \[
 T(t)f(\>a) = f(\>x(t,\>a)).
\]
Since $\>x(t+s,\>a) = \>x(t,\>x(s,\>a))$, it is not difficult to see
that $\{T(t),t\ge0\}$ is a strongly continuous contraction semi-group
on $\C(\LL)$ with generator, say $B$, which is closed by
Theorem~\ref{thm:closed}. Under condition \textbf{L}, $U(.)$ admits
continuous partial derivatives. Hence, the solution $\>x(t,\>a)$ of
system~\eqref{eq:diffx} is also twice differentiable with respect to
$\>a$ (see~\cite{Car}, Theorems 3.4.2 and 3.7.1, pp.~148 and 152
respectively). Then Theorem~\ref{thm:core} shows that $\C_1(\LL)$ is a
core for $B$ and $B=\overline{G}$.

The final line of argument proceeds in two steps.
\begin{enumerate}
\item Choose $f\in\C_2(\LL)$, so that $T(t)f(\>r(0))\in\C_2(\LL)$.
Since $\C_2(\LL)$ is dense in $\C(\LL)$, Theorem~\ref{thm:core}
entails that $\C_2(\LL)$ is a core for $B$.
\item Use equation~\eqref{eq:limgen} together with
Theorem~\ref{thm:convergence}.
\end{enumerate}
\end{proof}

\subsection{Some analytic facts}\label{sec:analytic}
Let, for $|z|\le1$,
\begin{eqnarray*}
r(t,z)=\sum_{i=1}^{\infty} r_i(t)z^i,  &\quad \DD\pi(t,z)= \sum_{i=1}^{\infty}\pi_i(\>r(t))z^i,  \\
\varphi(z)= \sum_{i=1}^{\infty}\varphi_iz^i, &\DD \theta(z)= \sum_{i=0}^{\infty}\theta_i z^i.
\end{eqnarray*}

In this respect, the model of \emph{Ma Micro Planète} presented in Section
\ref{sec:introduction} corresponds to
\[
\theta(z) = pz^{15} + (1-p)z^5, \quad \varphi(z)= z^{50}.
\]
Taking  the summation over $i$ in system~\eqref{eq:sys}, we get
immediately
\begin{equation}\label{eq:trans2}
\frac{\mathrm{d} r(t,z)}{\mathrm{d}t} + \mu \Bigl[1- \frac{1}{z}\Bigr] r(t,z) =\lambda \varphi(z) - \mu r_1(t)
+ \alpha [\theta(z)-1]\pi(t,z).
\end{equation}

The following proposition ensures in particular, for the class of
policies considered in Section~\ref{th:thermo}, the boundedness of the
Cesàro limit of $r(t,1)$ as $t\to\infty$, which, from its mathematical
definition, takes into account the POIs possibly going to infinity.

For any arbitrary positive function $f$, let $f^*$ denote its ordinary
Laplace transform
\[
f^*(s)\egaldef \int_0^{\infty} e^{-st}f(t)\mathrm{d}t, \quad \Re (s) \geq 0.
\]
\begin{prop}\label{prop:r1}
When $\{\theta_k,k\ge1\}$ is a proper probability distribution, so
that $\theta(1)=1$, we have
\begin{equation}\label{eq:bound1}
\lim_{t\to\infty} \frac{1}{t} \int_0^t r_1(t) \mathrm{d}t = \frac{\lambda}{\mu},
\end{equation}
Moreover, if $\theta'(1).\varphi'(1)<\infty$, and for any policy
satisfying Assumption~(\textbf{L}),
\begin{equation}\label{eq:bound2}
\lim_{t\to\infty} \frac{1}{t} \int_0^t r(t,1) \mathrm{d}t =  \frac{\alpha\underline{\pi}\theta'(1) + \lambda\varphi'(1)}{\mu}\, \le \, \frac{ \alpha \theta'(1) + \lambda \varphi'(1)}{\mu},
\end{equation}
where
\[
\underline{\pi} = \lim_{s\to0} s\pi^*\Bigl(s,\frac{\mu}{\mu+s}\Bigr)\le 1.
\]
\end{prop}

\begin{proof}
By a simple algebra carried out on equation~\eqref{eq:trans2}, we obtain
\begin{equation}\label{eq:trans3}
r^*(s,z) = \frac{1}{\mu\left(1-\frac{1}{z}\right)+s} \left[\frac{\lambda \varphi(z)}{s} +
\alpha\bigl(\theta(z)-1\bigr)\pi^*(s,z) -\mu r_1^*(s) + r(0,z)\right].
\end{equation}
Furthermore, applying Cauchy's formula
\[
r_1^*(s) = \frac{1}{2i\pi} \int_\C \frac{r^*(s,y)\mathrm{d}y}{y^2}
\]
in equation~\eqref{eq:trans3}, where $\C$ stands for the unit
circle centered at the origin, we get
\begin{equation*}
\begin{split}
r_1^*(s) & = \frac{1}{2i\pi} \int_\C \frac{\lambda \varphi(y) \mathrm{d}y}{sy\left[(\mu+s)y-\mu\right]} \\[0.2cm]
& +
\frac{1}{2i\pi} \int_\C\frac{\left[\alpha(\theta(y)-1)\pi^*(s,y) -\mu r_1^*(s)+r(0,y)\right]\mathrm{d}y}
{y[(\mu+s)y-\mu]},
\end{split}
\end{equation*}
which in turn, by a repeated (and elementary) application of Cauchy's
residue theorem, allows to extract
\begin{equation}\label{eq:r1trans}
r_1^*(s) =  \frac{\lambda \varphi\bigl(\frac{\mu}{\mu+s}\bigr)}{\mu s} + \frac{\alpha}{\mu}\left[\theta\Bigl(\frac{\mu}{\mu+s}\Bigr)-1\right] \pi^*\Bigl(s,\frac{\mu}{\mu+s}\Bigr) + \frac{1}{\mu}
r\Bigl(0,\frac{\mu}{\mu+s}\Bigr).
\end{equation}
Hence, instantiating the latter expression of $r_1^*(s)$ in
\eqref{eq:trans3}, we obtain finally
\begin{equation}\label{eq:trans4}
\begin{split}
r^*(s,z)  &  = \frac{1}{\mu\left(1-\frac{1}{z}\right)+s} \biggl[\frac{\lambda}{s}\biggl(\varphi(z) -
 \varphi\Bigl(\frac{\mu}{\mu+s}\Bigr)\biggr)\\
 &\quad + \alpha\bigl(\theta(z)-1\bigr)\pi^*(s,z) + \alpha\biggl(1-\theta\Bigl(\frac{\mu}{\mu+s}\Bigr)\biggr)\pi^*\Bigl(s,\frac{\mu}{\mu+s}\Bigr)\\
 &\quad +r(0,z) - r\Bigl(0,\frac{\mu}{\mu+s}\Bigr)\biggr].
\end{split}
\end{equation}
Take now $s>0$ and let $z\to1_-$ by real positive values in
\eqref{eq:trans4}. This yields
\begin{equation}\label{eq:trans5}
r^*(s,1) = \frac{1}{s} \biggl[\frac{\lambda}{s}\biggl(1 - \varphi\Bigl(\frac{\mu}{\mu+s}\Bigr) \biggr) +\alpha\biggl(1-\theta\Bigl(\frac{\mu}{\mu+s}\Bigr)\biggr)\pi^*\biggl(s,\frac{\mu}{\mu+s}\biggr)\biggr],
\end{equation}
where we have used the fact that $\pi^*(s,1)=1/s$, since  $\pi(t,1)=1, \forall t\geq0$.

Moreover, $\pi^*\Bigl(s,\frac{\mu}{\mu+s}\Bigr) < \pi^*(s,1)=
\frac{1}{s}$, for any $s>0$, and we are now in a position to apply a
Tauberian-type theorem for positive functions, see e.g.~\cite{FucLev},
just letting $s\to0$ by positive values in~\eqref{eq:trans5}, so that
\begin{align*}
\lim_{s\to0} s r^*(s,1) & = \lim_{t\to\infty} \frac{1}{t} \int_0^t r(t,1) \mathrm{d}t \\
& = \lim_{s\to0} \left[\frac{\lambda}{s}\biggl(1 - \varphi\Bigl(\frac{\mu}{\mu+s}\Bigr)\biggr)
+\frac{\alpha}{s}\biggl(1-\theta\Bigl(\frac{\mu}{\mu+s}\Bigr)\biggr)s\pi^*(s,1)\right] \\
& = \frac{\alpha \underline{\pi}\theta'(1) + \lambda \varphi'(1)}{\mu},
\end{align*}
which is exactly the asserted equality~\eqref{eq:bound2}. It is worth
noting that for the moment there is no need to assume the finiteness
of the derivatives $\theta'(1)$ and $\varphi'(1)$, but this will be
made more precise in Section~\ref{sec:stat}.

As for~\eqref{eq:bound1}, it is readily obtained from
\eqref{eq:r1trans}, using the inequality
$s\pi^*\Bigl(s,\frac{\mu}{\mu+s}\Bigr)\le1$, and again the above
Tauberian theorem. The proof of the proposition is concluded.
\end{proof}

\section{System behavior as
  \texorpdfstring{$t\to\infty$}{t->oo}}\label{sec:stat} 

In the sequel, we shall focus our attention solely on the
deterministic dynamical system $\>r(t)$ obtained in Theorem
\ref{thm:thermo}. Of special interest will be the non-degenerate
stationary limits
 \[
 \>r = \lim_{t\to\infty} \>r(t), \quad \mathrm{with} \  r_i(t)\to r_i,\forall i \ge 1,
 \]
and
\begin{equation}\label{eq:pibar}
\overline{\pi}\egaldef\lim_{z\to1_{-}}\lim_{t\to\infty}\pi(t,z)=\sum_{i\ge
  1}\pi_i(\>r).
\end{equation}

Indeed, from the general theory of differential systems (see e.g.,
\cite{Car,Ince}), under Assumption \textbf{(L)} and for a given set of
parameters, we know that a unique stationary regime of the system
exists, but \emph{possibly degenerate} in the sense that
$\overline{\pi}<1$. Note that this fact is not easy to prove by a
probabilistic argument, since, due to the normalizing condition,
system~\eqref{eq:sys} cannot be interpreted as forward Kolmogorov's
equations for a special birth and death process, where $r(t,z)/r(t,1)$
would represent the probability distribution function of the score of
an arbitrary POI.

 Throughout the rest of the paper, we shall speak of
 \emph{ergodicity}, with a slight abuse of language: this will always
 refer to a \emph{finite stationary} dynamical system satisfying the
 conditions
 \begin{equation}\label{eq:ergdef}
 r(1) = \sum_{i\ge1}r_i <\infty,\quad    \overline{\pi}=1.
\end{equation}
This notion of \emph{ergodicity} will be also referred to as a
\emph{non-herding} behavior, in agreement with the brief description
given in Section~\ref{sec:introduction}. Conversely, a \emph{herding}
behavior will refer to \emph{transient} phenomena, as introduced more
precisely in Section~\ref{sec:transience}.

\subsection{Global conservation equations}

Define the generating functions
\[
r(z)\egaldef \sum_{i=1}^{\infty}r_iz^i,\qquad \pi(z)\egaldef \sum_{i=1}^{\infty}\pi_i(\>r)z^i,
\]
with the shortened notation  $\pi(z) \equiv \pi(z;\>r)$, and
\begin{equation}\label{eq:def}
\frac{1-\theta(z)}{1-z} \egaldef \sum_{n\ge0} \Theta_n z^n, \quad
\frac{1-\varphi(z)}{1-z} \egaldef \sum_{n\ge0} \Phi_n z^n.
\end{equation}
Two characteristic values are of special interest, namely
\begin{itemize}
\item $r(1)=\sum_{i\ge 1}r_i$, the mean number of POIs per player;
\item $r_1$, since $\mu r_1$ is the global death rate of POIs per player.
\end{itemize}
By~\eqref{eq:trans2}, one sees easily that $r(z)$ satisfies the
functional equation
\begin{equation}\label{eq:stat}
\alpha\left[1-\theta(z)\right]\pi(z)
+\mu[1-z^{-1}]r(z)= \lambda \varphi(z)-\mu r_1,
\end{equation}
from which can be deduced the following corollary, which is the direct
continuation of Proposition~\ref{prop:r1}.
\begin{coro}\label{coro1}\mbox{ }
\begin{enumerate}
\item We have always $\DD\lim_{z\to1}(1-z)r(z)=0$ and $r_1=\DD\frac{\lambda}{\mu}$.
\item If in addition  $\theta'(1).\varphi'(1)<\infty$, then
\begin{equation}\label{eq:stat2}
\frac{r(z)}{z} = \frac{\alpha\Theta(z)\pi(z)}{\mu}+ \frac{\lambda \Phi(z)}{\mu},\quad \forall |z|\le1,
\end{equation}
\begin{equation}\label{eq:r1}
r(1)=\frac{\alpha\overline{\pi}\theta'(1)+\lambda \varphi'(1)}{\mu}.
\end{equation}
\end{enumerate}
\end{coro}
\begin{proof}
Immediate from Proposition~\ref{prop:r1} together with notation~\eqref{eq:def}.
\end{proof}

\noindent\textbf{Assumption (F)}
\emph{In the rest of the paper, unless otherwise explicitly mentioned,
  the product $\theta'(1).\varphi'(1)<\infty$ will always be assumed
  to be finite}.

\bigskip
From equation~\eqref{eq:stat2}, we get the immediate recursion, valid $\forall n\ge0$,
\begin{equation}\label{eq:rec0}
r_{n+1} = \frac{\alpha}{\mu} \sum_{i=0}^{n-1}\pi_{n-i}\Theta_i
         + \frac\lambda\mu\Phi_n.
\end{equation}

In the next sections, we examine in detail the effect of specific
choices of the policy $\pi$ on the behavior of the system. One of main
goals of the study is to find out \emph{the necessary and sufficient
  conditions} for the system to be \emph{ergodic} in the sense given
by the conditions~\eqref{eq:ergdef}.

\subsection{Selection of POIs according to their scores}\label{sec:poi-scores}
Hereafter, but in Section~\ref{sec:cumul}, we shall consider a set of
policies where the selection of a POI depends solely on its score.
Letting each score $i\ge1$ be associated with a positive weight $a_i$,
then we recover the policy $\textbf{P}_{\>a}$ announced in Assumption
(\textbf{L}) of Section~\ref{th:thermo}, namely
\begin{equation}\label{eq:ai}
\pi_i(\>r) = \frac{a_i r_i}{K},
\quad K\egaldef \sum_{j=1}^\infty a_j r_j,
\end{equation}
and~\eqref{eq:rec0} can be rewritten as
\begin{equation}\label{eq:rec1}
r_{n+1} = \frac{\alpha}{\mu K} \sum_{i=0}^{n-1} a_{n-i}r_{n-i}i\Theta_i
         + \frac\lambda\mu\Phi_n.
\end{equation}

In the context of \emph{Ma Micro Plan\`ete}, it seems quite realistic
to assume that the $a_i$'s are \emph{increasing with $i$}, since POIs
with high scores should be more attractive. In the sequel, we shall
mainly consider this situation of \emph{herding effect}.

Since the $a_i$'s are assumed to form an increasing sequence, there
exists $\lim_{i\to\infty} a_i$. When this limit is not finite, the
following simple interesting proposition holds.
\begin{prop}\label{prop:trans}
If $\lim_{i\to\infty} a_i =\infty$, then the system is non-ergodic.
\end{prop}
\begin{proof}
The relation (\ref{eq:rec1}) leads to the simple bound
\[
 r_{n+1}\ge \frac{\alpha}{\mu K} a_nr_n\Theta_{0} = \frac{\alpha}{\mu K} a_nr_n,
\]
so that, when the $a_i$'s are unbounded, system~\eqref{prop:trans}
admits no solution with $K<\infty$.

Nonetheless, it turns out there is in this case exactly one admissible
solution to~\eqref{prop:trans}, namely
\[
r_{n+1} =  \frac\lambda\mu\Phi_n, \quad \text{with}\quad K=\infty,
\]
which implies also (see~\eqref{eq:pibar}) $\overline{\pi}=0$. Here, it
is worth checking that $\pi(\>r)$ does not satisfy a uniform Lipschitz
condition, since
\[
\lim_{j\to\infty} \frac{a_j}{\sum_{i\ge1}a_ir_i} = \infty.
\]
\end{proof}

Consequently the only interesting case is when $\{a_i,i\ge1\}$ is a
\emph{non-decreasing sequence} tending to a \emph{finite limit}, which
ad libitum can be taken equal to $1$, since the $a_i$'s are defined up
to a constant.

\subsection{When scores of selected POIs are exactly increased by one}\label{sec:score1}
As it will appear in Sections~\ref{sec:general1} and~\ref{sec:matrix},
a complete solution to the non-linear equation (\ref{eq:rec1}) for a
general $\theta(z)$ seems to be technically almost untractable. In
this section, in order to get ideas about the general behavior of the
system, we analyze in detail the simple (but nonetheless not
elementary!) case $\theta(z)=z$, for which equation~\eqref{eq:rec1}
becomes the following first order recursive sequence
\begin{equation}\label{eq:rec2}
r_{n+1} = \frac{\alpha}{\mu K} a_nr_n + \frac\lambda\mu\Phi_n.
\end{equation}

A simple algebra carried out on this last equation leads to the formula
\begin{equation}\label{eq:sol2}
r_{n+1} = \frac{\lambda}{\mu}A_n
            \sum_{j=0}^n\frac{\Phi_{j}}{A_{j}}\Bigl(\frac{\alpha}{\mu
              K}\Bigr)^{n-j},\ \ n\ge 0,
\end{equation}
where
\[
A_n \egaldef \prod_{i=1}^n a_i .
\]
The constant $K$ can then be recovered as
\begin{align*}
K &= \sum_{i\ge1} a_ir_i
   = \frac{\lambda}{\mu}\sum_{i=0}^\infty
      a_{i+1}A_i\sum_{j=0}^i\frac{\Phi_{j}}{A_{j}}\Bigl(\frac{\alpha}{\mu
              K}\Bigr)^{i-j}\\
  &=\frac{\lambda}{\mu}\sum_{j=0}^\infty
  \frac{\Phi_j}{A_j\bigl(\frac{\alpha}{\mu
      K}\bigr)^{j}}\sum_{i=j}^\infty A_{i+1}\Bigl(\frac{\alpha}{\mu
      K}\Bigr)^{i}\\
  &=\frac{\lambda K}{\alpha}\sum_{j=0}^\infty
  \frac{\Phi_j}{A_j\bigl(\frac{\alpha}{\mu K}\bigr)^{j}}\sum_{i=j+1}^\infty A_{i}\Bigl(\frac{\alpha}{\mu K}\Bigr)^{i},
\end{align*}
which leads to the following implicit equation
\begin{equation}\label{eq:consistency}
\frac{\alpha}{\lambda} = F\Bigl(\frac{\alpha}{\mu K}\Bigr) ,
\end{equation}
allowing to determine $K$, after having set
\begin{equation}\label{eq:Fz}
\begin{cases}
\DD F(z) \egaldef \sum_{k=1}^\infty u_kz^k ,\\[0.5cm]
\DD u_k = \sum_{j=0}^\infty\frac{A_{j+k}\Phi_j}{A_j}, \quad \forall k\ge 1.
\end{cases}
\end{equation}
So, the problem of showing the existence of a \emph{strictly positive
  and finite constant} $K$ is exactly tantamount to finding a strictly
positive real number $x$, such that
\[
\frac{\alpha}{\lambda} = F(x)= \sum_{k=1}^\infty u_k x^k.
\]
Hence, in order to get effective ergodicity conditions, one has to
check the lower and upper bounds of $F(x)$, which is an increasing
function of $x$, for $x\ge0$. We note first that $F(x)$ is finite for
small enough $|x|$ if, and only if,
\begin{equation}\label{eq:F-finite1}
 u_1=\sum_{j=0}^\infty a_{j+1}\Phi_j<\infty.
\end{equation}
But, ex hypothesis, $A_{k+1}/A_k = a_k$ increases to $1$ from below,
so that $u_k$ is a monotone decreasing sequence and inequality
\eqref{eq:F-finite1} is plainly equivalent to
\begin{equation}\label{eq:F-finite}
\varphi'(1) = \sum_{j\ge0}\Phi_j < \infty.
\end{equation}
As for for the upper bound, we have the immediate inequality
\[
u_k\le \varphi'(1), \quad \forall k\ge1,
\]
and hence the radius of convergence of $F(z)$ is exactly equal to $1$.
Now, in order to decide whether the consistency relation
(\ref{eq:consistency}) admits a solution, the key point will be to
analyze the quantity
\[
M=\lim_{x\to1_-}F(x).
\]
There are two possibilities.
\begin{itemize}
\item[(i)] If $M$ is infinite, then for any value of $\lambda$,
$\alpha$, $\mu$, it is possible to find a unique admissible finite
$K$, and we shall say that the system is \emph{ergodic}.
\item[(ii)] Conversely, if $M<\infty$, then it will act as a limiting
value for $\alpha/\lambda$, and hence a non trivial \emph{ergodicity
  condition} takes place.
\end{itemize}
The situation is summarized in the following brief statement.
\begin{thm}\label{thm:ergo1}
 Assume that  $\DD\lim_{i\to\infty} a_i \nearrow 1$.
\begin{itemize}
\item[(i)] The system is never ergodic when $\varphi'(1)=\infty$.
\item[(ii)] When $\varphi'(1)<\infty$, the system is ergodic if, and only if,
\begin{equation}\label{eq:conderg}
\frac{\alpha}{\lambda}
\le \sum_{k=1}^\infty \sum_{j=0}^\infty\frac{A_{j+k}\Phi_j}{A_j},
\end{equation}
where the r.h.s.\ of the inequation above may be infinite.
\end{itemize}
\end{thm}
It is interesting to note a few not so intuitive facts related
to~\eqref{eq:conderg}: firstly, this condition does not depend on the
value of $\mu$, as long as $\mu>0$ obviously. Secondly, the ergodicity
range increases when the initial score increases in distribution ---
that is, when the $\Phi_i$'s increase. \bigskip

To propose a somehow concrete classification, let us concentrate for
the remainder of this section on the reasonably general set of sequences
satisfying, for some strictly positive constants $\gamma$ and $\nu$,
\begin{itemize}
\item $a_1>0$,
\item the $a_i$'s form an increasing sequence,
\item $a_i = 1-\gamma i^{-\nu}+\mathcal{O}\bigl(i^{-\nu-1}\bigr)$.
\end{itemize}

The following proposition describes the \emph{phase transition}
phenomenon that occurs when the parameters vary.
\begin{prop}\label{prop:ergo2}
Assume  the sequence $a_i$ satisfies the above conditions, and that
$\theta(z)=z$. Then,
\begin{itemize}
\item if $\nu>1$, or $\nu=1$ and $\gamma\le 1$, the system is always ergodic;
  \item if $\nu<1$, or $\nu=1$ and $\gamma>1$, there is a finite
  ergodicity bound $M$  if, and only if, $\sum_{j>0}j^{\nu+1} \varphi_j<\infty$.
\end{itemize}
\end{prop}
\begin{proof}
The first step is to obtain a convenient estimate for $A_n$ when $n$
is large. Using the property $0<a_1<a_i<1$, valid for all $i$, one gets
\[
\log A_n = \sum_{i=1}^n \log a_i = -\gamma \sum_{i=i}^ni^{-\nu}+\mathcal{O}(1).
\]

The finiteness of $M$ depends strongly on the properties of the series
$\sum A_n$, which can be deduced from the following estimates.
\begin{itemize}
\item If $\nu>1$, then the Dirichlet series $\sum_{i\ge1} i^{-\nu}$
converges to some positive number, so that the series $\sum A_n$
diverges, and then clearly $M=+\infty$.
\item If $\nu=1$, then $A_n\approx C_1n^{-\gamma}$. When $\gamma\le
1$, $\sum A_n$ still diverges and $M=+\infty$; conversely, when
$\gamma>1$,
\[
\frac{1}{A_j}\sum_{i=j+1}^\infty A_i\approx C_2j,
\]
and $M$ is finite whenever the second moment $\sum_{j>0}j^2 \varphi_j$
of the distribution $\{\varphi_k,k\ge1\}$ exists.
\item If $\nu<1$, then  $A_n  = \mathcal{\mathcal{O}}(e^{-\gamma
  n^{1-\nu}})$. Using classical techniques (Riemann sums and integration
by parts), one gets
\[
\sum_{i=j+1}^\infty e^{-\gamma  i^{1-\nu}}
\approx \int_{j}^\infty e^{-\gamma  x^{1-\nu}}dx
\approx \frac{j^\nu}{\gamma(1-\nu)} e^{-\gamma  j^{1-\nu}},
\]
and hence $M$ is finite whenever $\sum_{j>0}j^{\nu+1} \varphi_j$ is.
\end{itemize}
This concludes the proof of the proposition, the results of which will
be illustrated in Section~\ref{sec:aipower}.
\end{proof}

\subsection{Selecting POIs in terms of their cumulative score distribution}
\label{sec:cumul}
All models considered so far are based on exogenous
parameters $a_i$, supposed to be known and independent of $\>r(t)$. Actually, it turns out that
most of the computations  remain valid when these parameters are a
\emph{functional of the state} $\>r(t)$. We tackle hereafter such a
model, for which ergodicity conditions can be explicitly obtained.

The selection policies family we are interested in will be expressed as a
functional of the cumulative score distribution. Let
\[
P_i\egaldef\frac{1}{r(1)}\sum_{j=1}^ir_j
\]
be the proportion of POIs with score less of equal to $i$, and let $f$
be a positive convex function on $[0,1]$ such that $f(0)=0$, $f(1)=1$
and $f'(1)<\infty$. Then the probability of selecting a POI with a
score less or equal to $i$ is chosen to be $f(P_i)$ or, equivalently,
\begin{equation}\label{eq:cum}
\pi_i(\>r) = f(P_i)-f(P_{i-1}).
\end{equation}

An example of such a policy is given by the following simple
algorithm: $c>1$ POIs are selected uniformly, and the one with the
largest score will be visited. This model is reminiscent of the
queueing model of~\cite{VveDobKar}, which proves to be very efficient
in thermodynamical limit. The probability of selecting a POI with
score equal to $i$ is then given by~\eqref{eq:cum} when $f(x)=x^c$.

From a formal point of view, we remain in the framework of
Section~\ref{sec:poi-scores}, as it can be seen just by setting
\begin{equation}\label{eq:notation}
a_i=\frac{r(1)}{f'(1)r_i}\bigl[f(P_i)-f(P_{i-1})\bigr],
\quad K= \frac{r(1)}{f'(1)},
\end{equation}
and remarking that the $a_i's$ are no more exogenous constants as before.
To check they form an increasing positive sequence bounded by $1$, it
suffices to note that, since $f$ is a  convex function of $x$,
\[
\frac{r(1)}{r_i}\bigl[f(P_i)-f(P_{i-1})\bigr]
\le f'(P_i)
\le \frac{r(1)}{r_i}\bigl[f(P_{i+1})-f(P_{i})\bigr]
\le f'(1).
\]

While the equations are similar, the situation is nonetheless very
different from that encountered in Section~\ref{sec:poi-scores}. The
model is more difficult to solve, as the $\pi_i$'s do not have a
simple expression in terms of the score $i$, so that a formal solution
like~\eqref{eq:sol2} is not really usable; however, they can be
computed by recurrence, since the normalization constant $r(1)$ is
given by~\eqref{eq:r1}. As a consequence, the ergodicity condition is
easier to obtain and, interestingly enough, appears to be
\emph{insensitive} to the distributions $\theta$ and $\varphi$.

\begin{thm}
Assume that the POIs are visited according to
policy~\eqref{eq:cum}, with $f$ strictly convex, and therefore
$1<f'(1)<\infty$. Then, according to the definition given in Section
\ref{sec:stat}, the system is ergodic if, and only if,
\[
\frac{\alpha\theta'(1)}{\lambda\varphi'(1)} \le \frac{1}{f'(1)-1}.
\]
\end{thm}
\begin{proof}
The recurrence relation~\eqref{eq:rec0} takes the form
\[
 P_{n+1}-P_n=\frac{\alpha}{\mu r(1)}
    \sum_{i=0}^{n-1}\bigl(f(P_{n-i})-f(P_{n-i-1})\bigr)\Theta_i
    +\frac{\lambda}{\mu r(1)}\Phi_n,
\]
which by a direct summation yields
\begin{equation}\label{eq:P}
 P_{n+1}=\frac{\alpha}{\mu r(1)}\sum_{i=0}^{n-1}f(P_{n-i})\Theta_i
   +\frac{\lambda}{\mu r(1)}\sum_{i=1}^n\Phi_i,\quad\forall\,n\ge0.
\end{equation}

The form~\eqref{eq:P} allows for a direct graphical analysis of the
convergence problem, starting from the following simple facts. Using
\eqref{eq:r1} with $\overline{\pi}=1$ , it is straightforward to check that the
non-decreasing sequence defined by~\eqref{eq:P} is bounded by $1$.
Therefore it always converges to a finite value
\[
  P_\infty\egaldef \lim_{n\to\infty}P_n \le 1,
\]
that satisfies, by a direct application of the celebrated Toeplitz lemma,
\[
  P_\infty=\frac{\alpha}{\mu r(1)}\theta'(1)f(P_\infty)+\frac{\lambda}{\mu r(1)}\varphi'(1).
\]

It is actually possible for the sequence to converge to some
$P_\infty<1$ but, $f(.)$ being a convex function, this can only
happen when the slope at the point $1$ is greater than $1$. This concludes
the theorem, since~\eqref{eq:r1} implies the equivalence
\[
\frac{\alpha\theta'(1)}{\mu r(1)}f'(1) \le 1 \quad\Longleftrightarrow\quad \frac{\alpha\theta'(1)}{\lambda\varphi'(1)} \le \frac{1}{f'(1)-1}.
\]
\end{proof}

\subsection{About transience: herding phenomena} \label{sec:transience}
Hereafter, we do not pretend to give exhaustive rigorous proofs of all our
claims, which can rather be viewed as \emph{very likely true}
conjectures, and have been verified by several numerical simulation
runs.

Our concern is to extend Theorem~\ref{thm:ergo1}, when there is no
solution to~\eqref{eq:conderg}, i.e.\ when its right-hand side member
is finite and $\alpha$ is too large. In fact, it appears that the
system has still a stationary regime, but a \emph{finite number of
  POIs have an infinite score}. This means that, as $t\to\infty$, the
selection policy $\pi(\>t)$ becomes \emph{defective}, namely (see
equation~\eqref{eq:pibar}) when
\[
\overline{\pi}<1.
\]
This implies exactly
\[
\pi_i(\>r)= \frac{a_ir_i}{\sum_{j\ge1}a_jr_j + \delta},
\]
where $\delta$ is a positive constant representing the proportion of
POIs going to infinity. Obviously, $\delta=0$ when the system is
\emph{normally} ergodic, i.e.\ no escape of mass to infinity. By
contrast, $\delta>0$ corresponds to a phase transition with
condensation. Then the consistency equation~\eqref{eq:consistency}
must be modified and becomes
\begin{equation}\label{eq:transient1}
1= \frac{\lambda}{\alpha} F\biggl(\frac{\alpha}{\mu K}\biggr) + \frac{\delta}{K},
\end{equation}
which has \emph{always} a solution $(K,\delta)$ by monotonicity with
respect to $K$, remembering that $F(1)<\infty$, while
$F(1+\varepsilon)=\infty, \forall\varepsilon>0$, and $\alpha >\lambda
F(1)$. Equation~\eqref{eq:transient1} is not sufficient to determine
the two unknowns $K$ and $\delta$. To get a second equation, we resort
to the following heuristic \emph{stochastic least action principle},
which in our opinion should be provable in a very general context.

\begin{ansatz} \emph{\textbf{The least action principle}}. When a
multicomponent irreducible Markovian system ceases to be ergodic, the
first component(s) becoming infinite with positive probability can be
identified by \emph{continuity} with respect to the set of parameters
$\mathcal{P}$ defining the system.
\end{ansatz}

This definition, which might look obscure to the reader, can be
rephrased by saying that the phase transition takes place as soon as
one touches some regions (hyperplanes or surfaces) in $\mathcal{P}$.
This is for instance the case for the famous Jackson-Kelly stochastic
queueing networks\ldots

\smallskip Applying this principle leads to say that we are looking
for the minimal $\delta>0$ ensuring the system
\[
 \>r = \lim_{t\to\infty} \>r(t)
\]
is no more ergodic, so that \emph{some POIs go to infinity}. This
implies necessarily
\begin{equation}\label{eq:transient2}
\frac{\alpha}{\mu K} = 1.
\end{equation}
Then~\eqref{eq:transient1} and~\eqref{eq:transient2} yield at once
\[
\begin{cases}
\DD\delta=\frac{\alpha-\lambda F(1)}{\mu}, \\[0.3cm]
\DD\overline{\pi}= \frac{K-\delta}{K}= \frac{\lambda F(1)}{\alpha} <1,\\[0.3cm]
\DD r(1)= \frac{\lambda}{\mu} [F(1) + \varphi'(1)].
\end{cases}
\]
A possible way to decide analytically when $\overline{\pi}<1$ would be
to design a converging iterative scheme to solve system~\eqref{eq:sys}
or~\eqref{eq:trans2} (see e.g.\ the polling network analyzed in
\cite{DelFay}): this appears to be another interesting, but
practically intricate problem.

\section{Some simple examples and limit cases} \label{sec:examples}
For the sake of completeness, we quote hereafter simple peculiar
cases, which hopefully shed light on some aspects of the behavior of
the system.

\subsection{No creation of POIs, i.e.\ \texorpdfstring{$\lambda=0$}{lambda=0}}
In this case, the system is empty at steady state for any Markovian
policy. The argument is elementary, since the vector state $\>0$ is
absorbing due to the $M/M/\infty$ character of the score decreasing
process.
\subsection{The lower bound \texorpdfstring{$\alpha=0$}{alpha=0}}
Here no POI will see any increase of score upon being visited, and
hence one gets a clear lower bound for the model. Solving
\eqref{eq:stat2} leads to
\[
 r_i=\frac{\lambda}{\mu}\sum_{k\ge i}\varphi_k,\quad r(1)= \frac{\lambda \varphi'(1)}{\mu}, \quad
 r'(1) = \frac{\lambda\bigl [\varphi''(1)+2\varphi'(1)\bigr]}{2\mu}.
\]
Consequently, when $\alpha\ne0$, we have the following bound, valid
for any point selection scheme (i.e.\ policy) $\pi$,
\[
 r'(1)\ge \frac{\lambda\bigl [\varphi''(1)+2\varphi'(1)\bigr]}{2\mu},
\]
and therefore the mean score $r'(1)/r(1)$ of a POI can be finite only when
$\varphi''(1)$ is finite. This leads to some unintuitive behavior, since it
was shown in Proposition~\ref{prop:ergo2} that having $\varphi''(1)=\infty$ can
at the same time make the system always ergodic.

\subsection{Score independent visits}
When the POIs are visited uniformly regardless of their score, which
amounts to the choice
\[
 \pi_i(\>r)=\frac{r_i}{\sum_{i\ge 1}r_i},
\]
we have $\pi(z)=r(z)/r(1)$. Then equation~\eqref{eq:stat2} allows to
derive immediately the corresponding generating function, which takes
the compact explicit form
\[
 r(z)= \frac{\lambda r(1)(\varphi(z)-1)}{\alpha (1-\theta(z))+\mu r(1)(1-z^{-1})},
\]
where, for $\lambda\ne0$,
\[
r(1) = \frac{\alpha \theta'(1) + \lambda \varphi'(1)}{\mu},
\]
and one can check the denominator of $r(z)$ does not vanish for
$|z|<1$. Moreover, $r(1)$ is clearly finite and given by equation
\eqref{eq:r1}. So, we have obtained a unique function $r(z)$ analytic
in the unit disc and continuous for $|z|=1$. Consequently the system
admits of a unique limit invariant measure $\{r_i,i\ge1\}$ \emph{for
  any value of the parameters}. Intuitively, this is due to the fact
that the decrease rate of the scores is somehow proportional to the
number of points, like in a $M/M/\infty$ queue. But, when the
selection process favors heavily enough the POIs having a big score,
proper ergodicity conditions appear, as in the models analyzed in
Section~\ref{sec:poi-scores}.
\subsection{An unstable  model with a score dependent policy} \label{sec:notreal}

A policy that may seem attractive at first sight is to give to each
POI a weight proportional to its score
\[
  \pi_i(\>r)=\frac{ir_i}{\sum_{i\ge 1}ir_i},
\]
so that $\pi(z)= zr'(z)/r'(1)$ and~\eqref{eq:stat2} becomes a
differential equation
\begin{equation*}
\alpha\left[1-\theta(z)\right]\frac{zr'(z)}{r'(1)}
+\mu [1-z^{-1}]r(z)=\lambda [\varphi(z)-1].
\end{equation*}

However, the weight $a_i=i$ given to the score $i$ is unbounded, and
from Proposition~\ref{prop:trans} we conclude that no choice of parameters can
lead to a stable system under this policy. More precisely, \emph{there
  exists at least one POI having an infinite score}.
\subsection{A case study of Proposition~\ref{prop:ergo2}}\label{sec:aipower}
Assuming the conditions of Proposition~\ref{prop:ergo2} hold, let us take in
addition $\varphi(z)=z$ and, for some $\gamma>0$,
\[
  a_i = \Bigl(\frac{i}{i+1}\Bigr)^\gamma.
\]

This implies $A_k=1/(k+1)^\gamma$ and therefore
(\ref{eq:sol2})--(\ref{eq:consistency}) become
\begin{align*}
 r_{k+1}&=\frac{\lambda}{\mu}\frac{1}{(k+1)^\gamma}\Bigl(\frac{\mu K}{\alpha}\Bigr)^k,\\
 \frac{\alpha}{\lambda}&=F(x)=\sum_{k=1}^\infty A_k x^k=\Phi(x,\gamma,1)-1,
\end{align*}
where $\Phi$ stands for the classical Lerch function~\cite[formula
9.550]{GraRyz}
\[
 \Phi(x,\gamma,v)\egaldef \sum_{k=0}^\infty\frac{x^k}{(k+v)^\gamma}.
\]

When $\varphi(z)=z$, the computation of the expected score of POIs per
player $r'(1)$ is easy:
\[
 r'(1)
=\sum_{k=1}^\infty k r_k
= \sum_{k=1}^\infty \frac{\lambda}{\mu}\sum_{k=1}^\infty
k A_k\Bigl(\frac{\alpha}{\mu K}\Bigr)^k
= \frac{\lambda}{\mu}\frac{\alpha}{\mu K}F'\Bigl(\frac{\alpha}{\mu K}\Bigr).
\]
\begin{figure}
\centering\includegraphics[angle=-90,width=0.8\linewidth]{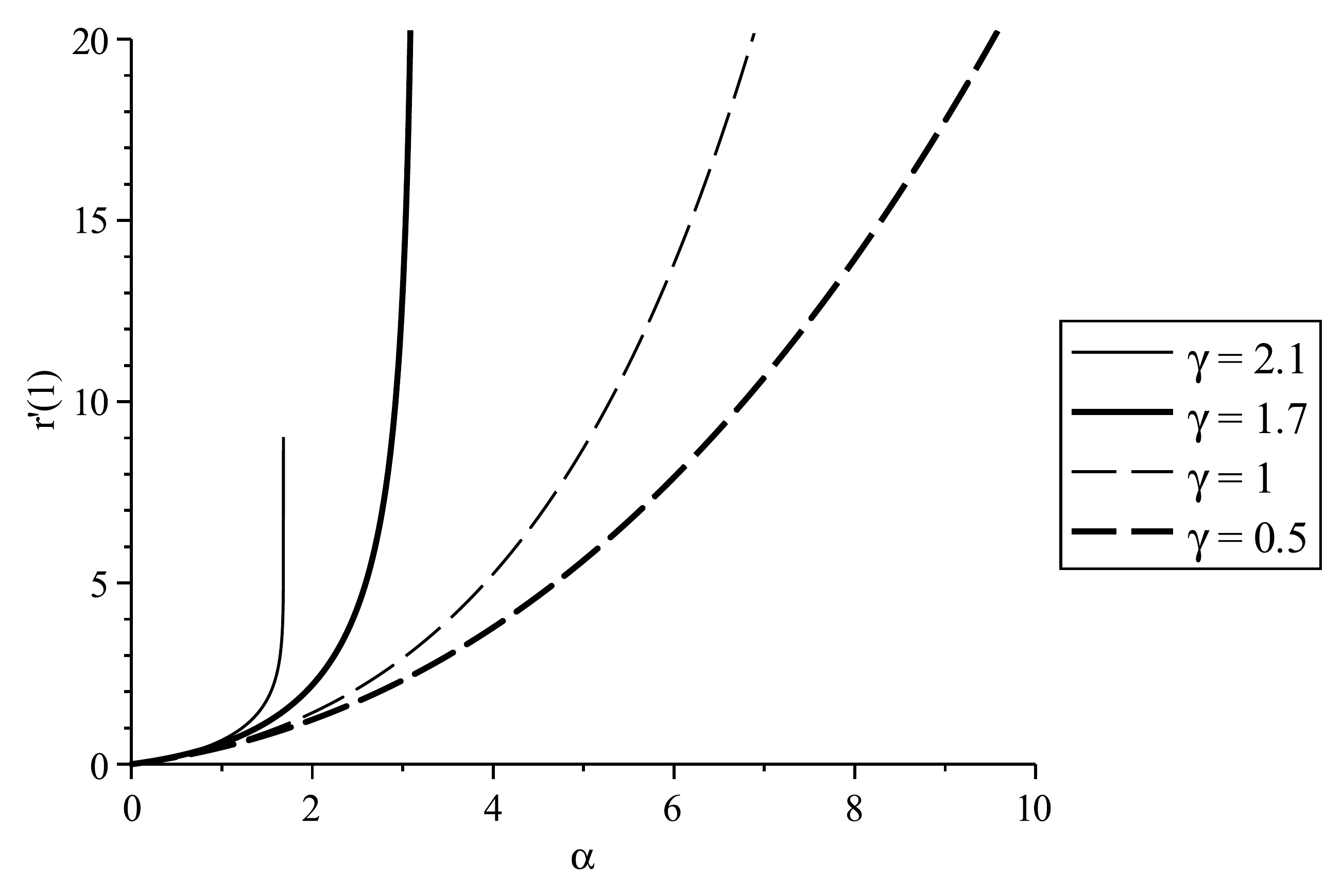}
\caption{Expected score of POIs per player $r'(1)$ as a function of
  the visit rate $\alpha$. The parameters $\lambda=\mu=3$ remain
  fixed.}\label{fig:curves}
\end{figure}

The above formulas allow to represent the expected score per player
as a function of $\alpha$ in a parametric way, by letting $x$ vary
from $0$ to $1$ (Figure~\ref{fig:curves}). As expected from the
theoretical results obtained above, $3$ different situations prevail:
\begin{itemize}
\item $\gamma\le 1$: the system is always ergodic (dashed
curves);
\item $1<\gamma\le 2$: there is an explicit ergodicity condition on $\alpha$,
and $r'(1)$ tends to infinity at the boundary;
\item $\gamma>2$: there is an explicit ergodicity condition on $\alpha$,
but $r'(1)$ remains bounded.
\end{itemize}

\section{Miscellaneous outcomes and open problems} \label{sec:general}
The next three paragraphs present briefly some ideas and mathematical
questions concerning the general model.
\subsection{General approach via an integral equation}\label{sec:general1}
Consider the situation of Section~\ref{sec:poi-scores}, and let $a(z)
\egaldef \sum_{j\ge1}a_jz^j$. Denote the respective convergence radii
of $r(\cdot)$ and $a(\cdot)$ by $\rho_r$ and $\rho_a$. Then, by
Hadamard's multiplication theorem~\cite{Tit:1} and definition
\eqref{eq:ai}, we can write
\[
\pi(z) = \frac{1}{2i\pi K}\int_\mathcal{L} r(w)a\left(\frac{z}{w}\right) \frac{\mathrm{d}w}{w},
\]
where $\mathcal{L}$ stands for a circle of radius $\rho$, centered at
the origin, such that $|z|/\rho_a<\rho<\rho_r$. This yields the
non-linear integral equation
\begin{equation}\label{eq:intequ}
\frac{\mu r(z)}{z} = \frac{\alpha[1-\theta(z)]}{2i\pi K(1-z)} \int_\mathcal{L} r(w)a\left(\frac{z}{w}\right) \frac{\mathrm{d}w}{w} + \frac{\lambda(1-\varphi(z))}{1-z}.
\end{equation}
Methods of solution of~\eqref{eq:intequ} (e.g.\ via boundary value
problem or Fredholm integral equation) depend heavily on the explicit
form of $a(\cdot)$. Alas, tractable explicit expressions can hardly be
derived (nor expected!), even when $a(\cdot)$ is a polynomial.

\subsection{General approach via a matrix form solution}\label{sec:matrix}
The problem becomes more intricate when $\theta(z)$ is an arbitrary
probability generating function with $\varphi'(1)<\infty$. Then the
following approach, based on the recursion~\eqref{eq:rec1}, involves
infinite matrices and seems to be computationally quite acceptable. We
sketch hereafter its main lines.

Clearly, equation~\eqref{eq:rec1} can be rewritten in the vector form
\begin{equation}\label{eq:vector}
\>U_{n+1} = M_n\Bigl(\frac{\alpha}{\mu K}\Bigr)\>U_n +\frac{\lambda}{\mu}\>V_{n+1},
\end{equation}
where
\begin{itemize}
\item $\>U_n$ is the $n$-column vector $[r_n r_{n-1} \ldots r_1]^T$, where ${}^T$ denotes the transpose operation;
\item $\>V_n$ is the $n$-column vector $[\Phi_{n-1}, 0, \dots,0]^T$;
\item $M_n(x)$ is a $(n+1)\times n$ matrix for all $x>0$. Its first
row is the vector $x\>B_n$, where $\>B_n$ denotes the row vector
\[
[a_n\Theta_0, a_{n-1}\Theta_1, \dots,a_1\Theta_{n-1}],
\]
and the remaining internal matrix is exactly the $n\times n$ identity
matrix denoted by $I_n$.
\end{itemize}
By a direct algebra, the solution takes the form
\begin{equation}\label{eq:vector2}
\>U_{n+1} =  \frac{\lambda}{\mu}\sum_{i=1}^{n+1} Q_{n,i}\Bigl(\frac{\alpha}{\mu K}\Bigr)\>V_i , \quad n\ge0,
\end{equation}
where
\[
\begin{cases}
Q_{n,i}(x) = M_n(x)M_{n-1}(x)\ldots M_i(x), \quad i\le n, \\[0.5cm]
Q_{n,n+1}(x) \egaldef I_{n+1}.
\end{cases}
\]
The next step is to carry out the scalar product of~\eqref{eq:vector2}
with the row vector
 \[
 \>A_{n+1}\egaldef [a_{n+1},\ldots,a_1],
\]
and then to let $n\to\infty$. Keeping in mind that $K= \sum_{i\ge1}
a_i r_i$, we get the following final relationship, which is formally
similar to the one given by~\eqref{eq:consistency},
\begin{equation}\label{eq:final}
 \frac{\alpha}{\lambda} = \frac{\alpha}{\mu K}\sum_{n=0}^\infty\sum_{i=1}^{n+1}  \>A_{n+1}\widetilde{Q}_{n,i}\Bigl(\frac{\alpha}{\mu K}\Bigr)\>V_i .
\end{equation}
From the definition of $M_n$, each term of the double sum in
\eqref{eq:final} is clearly a \emph{decreasing} scalar function of
$K$. Hence, ergodicity conditions could be obtained along the same
lines as in Theorem~\ref{thm:ergo1}, by analyzing the series in the
right-hand side member of~\eqref{eq:final}.

\subsection{The case \texorpdfstring{$\theta'(1).\varphi'(1)=\infty$}{theta'(1)phi'(1)=oo}}
When either $\theta'(1)$ or $\varphi'(1)$ are infinite, we are a
priori facing a \emph{transient} phenomenon as $r(1)=\infty$. However,
the system can have an interesting behavior provided that $K=
\sum_{i\ge1} a_i r_i<\infty$, which yields necessarily
\[
\liminf_{i\to\infty}a_i = 0,
\]
but this would correspond to a \emph{non-herding} system, which is an
other story, although most of the mathematical arguments presented in
Section~\ref{sec:stat} could be readily applied.

\appendix
\section{Generators, cores and weak convergence}\label{secC}
Here we quote the material necessary for the proof of Theorem
\ref{th:thermo}. The results are borrowed from~\cite{EK}.
\begin{thm}\label{thm:closed}
If $O$ is the generator of a strongly continuous semigroup $\{T(t)\}$
on $\L$, then its domain $\D(O)$ is dense in $\L$ and $O$ is closed.
\end{thm}
\begin{defi} \textup{\cite[p.~17]{EK}}
Let $O$ be a closed linear operator with domain $\D(O)$. A subspace
$S$ of $\D(O)$ is said to be a \emph{core} for $O$ if the closure of
the restriction of $O$ to $S$ is equal to $O$, i.e., if
$\overline{O_{|S}}=O$.
\end{defi}
The next proposition is an important criterion to
characterize a core.

\begin{thm}\label{thm:core} \textup{\cite[p.~17]{EK}} Let $O$ be the generator of a strongly
continuous contraction semigroup $\{T(t)\}$ on $\L$. Let $\D_0$ and
$\D$ be dense subspaces of $\L$ with $\D_0\subset\D\subset \D(O)$. If
$T(t): D_0\to\D$ for all $t\ge0$, then $\D$ is a core for $O$.
\end{thm}
The proof of Theorem~\ref{thm:thermo} relies heavily on the next general
proposition.
\begin{thm}\label{thm:convergence} \textup{\cite[Theorem 6.1, p.~28]{EK}}
In addition to $\L$, let $\L_k, k\geq 1$, be a sequence of Banach
spaces, $\Pi_k : \L\rightarrow\L_k$ be a bounded linear
transformation, subject to the constraint $\sup_k
\|\Pi_k\|<\infty$. Let also $\{T_k(t)\}$ and $\{T(t)\}$ be strongly
continuous contraction semigroups on $\L_k$ and $\L$ with generators
$O_k$ and $O$. We write $f_k\rightarrow f$ to mean exactly
 $$f\in\L, \quad f_k\in\L_k \quad \mathrm{for} \ k\geq 1, \quad
\mathrm{and}\ \lim_{k\rightarrow\infty}\|f_k - \Pi_kf\|=0 .$$ Then, if
$D$ is a core for $O$, the following are equivalent:
\begin{itemize}
\item[{\rm (a)}] For each $f\in\L$, $T_k(t)\Pi_kf\rightarrow T(t)f$ for all
$t\geq 0$, uniformly on bounded intervals.
\item[{\rm (b)}] For each $f\in\L$, $T_k(t)\Pi_kf\rightarrow T(t)f$ for all
$t\geq 0$.
\item[{\rm (c)}] For each $f\in D$, there exists $f_k\in\D(O_k)$ for
each $k\geq 1$, such that $f_k\rightarrow f$ and
$O_kf_k\rightarrow Of$.
\end{itemize}
\end{thm}


\end{document}